\documentclass[10pt,reqno]{amsproc}

\usepackage{hyperref}
\usepackage{amsthm,latexsym,amssymb,amsmath}
\usepackage{esint,mathptmx}
\newcommand{\E}{\mathrm{E}}
\newcommand{\OO}{\mathrm{O}}
\renewcommand{\H}{\mathcal{H}}
\newcommand{\J}{\mathcal{J}}
\newcommand{\T}{\mathcal{T}}
\renewcommand{\S}{\mathcal{S}}
\newcommand{\U}{\mathcal{U}}
\newcommand{\D}{\mathcal{D}}

\newcommand{\R}{\mathbb{R}}
\newcommand{\C}{\mathbb{C}}
\newcommand{\N}{\mathbb{N}}
\newcommand{\<}{\left\langle}
\renewcommand{\>}{\right\rangle}
\renewcommand{\phi}{\varphi}

\newcommand{\Id}{\,\mathrm{Id}}

\newcommand{\scal}{\,\mathrm{scal}}
\newcommand{\Ric}{\mathrm{Ric}}
\newcommand{\eps}{\varepsilon}
\newcommand{\Dg}{\,\mathcal{D}\gamma}
\newcommand{\Hom}{\mathrm{Hom}}
\newcommand{\id}{\mathrm{id}}

\newcommand{\Pexp}{\mathrm{P^{\nabla}exp}}
\newcommand{\PP}{\mathcal{P}}

\newcommand{\PMxy}{{\mathfrak{P}(M)_x^y}}
\newcommand{\PMt}{{\mathfrak{P}(M,t)}}
\newcommand{\PMtx}{{\mathfrak{P}(M,t)_x}}
\newcommand{\PMty}{{\mathfrak{P}(M,t)^y}}
\newcommand{\PMtxy}{{\mathfrak{P}(M,t)_x^y}}
\newcommand{\PPM}{{\mathfrak{P}(\PP,M)}}
\newcommand{\PPMx}{{\mathfrak{P}(\PP,M)_x}}
\newcommand{\PPMy}{{\mathfrak{P}(\PP,M)^y}}
\newcommand{\PPMxy}{{\mathfrak{P}(\PP,M)_x^y}}
\newcommand{\Mr}{M^{\times r}}
\newcommand{\hme}{^{-1}}
\newcommand{\kD}{k^\Delta}
\newcommand{\kH}{k^H}
\newcommand{\kHN}{k^H_{(N)}}
\newcommand{\ki}{k^1}
\newcommand{\kii}{k^2}
\newcommand{\kiii}{k^3}
\newcommand{\kiv}{k^4}

\newtheorem{thm}{Theorem}
\newtheorem{cor}{Corollary}
\newtheorem{prop}{Proposition}
\newtheorem{lemma}{Lemma}

\theoremstyle{definition}
\newtheorem{ex}{Example}
\newtheorem{rem}{Remark}
\newtheorem{definition}{Definition}

\long\def\symbolfootnote[#1]#2{\begingroup%
\def\thefootnote{\fnsymbol{footnote}}\footnote[#1]{#2}\endgroup} 

\title[Renormalized Integrals and a Path Integral Formula for the Heat Kernel]{Renormalized Integrals and a Path Integral Formula for the Heat Kernel on a Manifold}
\author{Christian B\"ar}
\address{Universit\"at Potsdam, Institut f\"ur Mathematik, Am Neuen Palais 10, 14469 Potsdam, Germany}
\email{baer@math.uni-potsdam.de}
\urladdr{http://geometrie.math.uni-potsdam.de/}

\dedicatory{Dedicated to Prof.\ Steven Rosenberg on the occasion of his sixtieth birthday}
\thanks{The author thanks the \emph{Sonderforschungsbereich 647}, funded by \emph{Deutsche Forschungsgemeinschaft}, for financial support.}

\begin{document}

\maketitle

\begin{abstract}
We introduce renormalized integrals which generalize conventional measure theoretic integrals.
One approximates the integration domain by measure spaces and defines the integral as the limit of integrals over the approximating spaces.
This concept is implicitly present in many mathematical contexts such as Cauchy's principal value, the determinant of operators on a Hilbert space and the Fourier transform of an $L^p$-function.

We use renormalized integrals to define a path integral on manifolds by approximation via geodesic polygons.
The main part of the paper is dedicated to the proof of a path integral formula for the heat kernel of any self-adjoint generalized Laplace operator acting on sections of a vector bundle over a compact Riemannian manifold.
\end{abstract}

\section{Introduction}

Path and functional integrals are an important tool in quantum field theory but in many cases a solid mathematical foundation is still lacking.
In some cases one knows that the desired integral cannot be realized by a conventional integral because the necessary measure cannot exist.
We propose a mathematical framework that might be able to deal with this difficulty.
We call it \emph{renormalized integrals}.
The idea is that the measure on the space (or even the space itself) over which we want to integrate might not exist but we can approximate it by measure spaces and then define the integral as the limit of the conventional integrals on these measure spaces.
Details and examples are given in Section~\ref{sec:RenInt}.
It turns out that a variety of mathematical concepts can be regarded as renormalized integrals such as Cauchy's principal value, the determinant of operators on a Hilbert space and the Fourier transform of an $L^p$-function.

We then concentrate on the space of paths on a compact Riemannian manifold.
We approximate it by spaces of piecewise geodesics and define the path integral as a renormalized integral.
The functions we would like to path integrate are induced by integral kernels on the manifold.
This is explained in Example~\ref{ex:qQ}.
We call two integral kernels \emph{heat-related} if their difference can be estimated by the heat kernel for the Laplace-Beltrami operator.
The central result of Section~\ref{sec:PathIntegrals} is Proposition~\ref{prop:replace}, where we show that heat-related integral kernels give rise to functions on path space with equal path integrals.
Lemma~\ref{lem:gaussest} is a useful criterion that allows one to prove that two integral kernels are heat-related.

In Theorem~\ref{thm:heatkernel} we give a path integral formula for the heat kernel of any self-adjoint generalized Laplace operator acting on sections of a vector bundle over the manifold.
This improves the results of \cite{BP} (and also those of the earlier article \cite{AD} which deals with scalar operators) where one has a similar formula for the solutions to the heat equation but only weaker results for the heat kernel itself.
This difference is subtle; it is analogous to the passage from Brownian motion to the Brownian bridge in stochastic analysis.
The present paper refines the analysis in \cite{BP}.
The concept of Chernoff equivalence used in \cite{BP,SWW03,SWW07} had to be replaced by more refined ``heat bounds''.

It is also possible to write the heat kernel as an integral over path space equipped with the Wiener measure.
This is known as the Feynman-Kac formula.
It has the advantage that the whole machinery of measure and integration theory and stochastic analysis can be applied, see e.g.\ \cite{BP2,Bismut,Elworthy,Emery,Gangolli,HT,H,Stroock} for this approach.
But it also has disadvantages: part of the function that one wants to integrate over path space gets swallowed by the Wiener measure.
For this reason one can show that the Wiener measure cannot be modified in such a way that one can write the solutions to the Schr\"odinger equation as a path integral; see \cite[Sec.~4.6]{JL}.
This is a serious drawback because the Schr\"odinger equation in quantum mechanics was Feynman's original concern when he invented his path integral approach which has turned out to be so influential in theoretical physics to date.
There is promising indication that renormalized integrals will be able to deal with the Schr\"odinger equation, but the analytic details still have to be worked out.

\section{Renormalized integrals}
\label{sec:RenInt}

We start by describing the abstract concept of renormalized integrals.
Let $\J$ be a directed system, i.e., $\J$ is a set equipped with a relation $\preceq$ such that the following holds:
\begin{itemize}
\item 
Reflexivity $\quad\forall\, \T\in\J:\quad \T\preceq \T$
\item
Transitivity $\quad\forall\, \T,\S,\U\in\J:\quad \T\preceq\S\,\, \&\,\, \S\preceq\U\,\, \Rightarrow\,\, \T\preceq \U$ 
\item
Antisymmetry $\quad\forall\, \T,\S\in\J:\quad \T\preceq\S\,\, \&\,\, \S\preceq\T\,\, \Rightarrow\,\, \T=\S$
\item
$\forall\, \T,\S\in\J\,\, \exists\, \U\in\J:\quad \T\preceq\U\,\, \&\,\, \S\preceq \U$
\end{itemize}
\noindent
We call a family of measure spaces $\Omega = \{(\Omega_\T,\mu_\T)\}_{\T\in\J}$ parameterized by $\J$ a \emph{measure space family}.
We think of $\Omega$ as a virtual space that is approximated by the measure spaces $(\Omega_\T,\mu_\T)$.
Let $(X,\|\cdot\|)$ be a Banach space.

\begin{definition}\label{def:integrierbar}
By a \emph{measurable function} on a measure space family $\Omega$ with values in $X$ we mean a family $f=\{f_\T\}_{\T\in\J}$ of measurable functions $f_\T:\Omega_\T \to X$.
By abuse of notation, we write $f:\Omega \to X$ and think of $f$ as a function on the virtual space $\Omega$.
We call $f$ \emph{integrable} if $f_\T$ is eventually integrable\symbolfootnote[2]{i.e., $\exists\,\S\in\J:\,\,f_\T$ is integrable for all $\S\preceq\T$.} and the limit
\[
\fint_\Omega f(x)\D x := \varinjlim_{\T\in\J} \int_{\Omega_\T}f_\T(x)\,d\mu_\T(x) \in X
\]
exists.
In other words, we demand that the values of the integrals converge in the sense of nets.
We then call the limit $\fint_\Omega f(x)\D x$ the \emph{renormalized integral} of $f$ over $\Omega$.

Similarly, we can define measurable and integrable functions and their renormalized integrals when they take values in $\overline\R = [-\infty,\infty]$.
\end{definition}

\begin{ex}\label{ex:renint1}
Let $\J=\R_+=(0,\infty)$ and ``$\preceq$'' $=$ ``$\leq$''.
For $T\in\R_+$ let $\Omega_T = [-T,T] \subset \R$ and $\mu_T = \frac{1}{2T}\, \times $ Lebesgue measure.
Given a measurable function $f:\R\to\R$ we get a measurable function on $\Omega$ in the sense described above by putting $f_T := f|_{\Omega_T}$.
For example, for $f\equiv 1$ we have
\[
\int_{\Omega_T}f(x)\,d\mu_T(x) = \frac{1}{2T} \int_{-T}^T f(x)\,dx = 1
\]
and hence
\[
\fint_\Omega 1\, \D x = \lim_{T\to\infty} \int_{\Omega_T}f(x)\,d\mu_T(x) = 1.
\]
More generally, let $\alpha>-1$ and $f(x) = (|x|+1)^\alpha$.
Then
\[
\int_{\Omega_T}f(x)\,d\mu_T(x) = \frac{1}{2T} \int_{-T}^T (|x|+1)^\alpha\,dx = \frac{1}{T}\frac{1}{\alpha+1}\left((T+1)^{\alpha+1}-1\right)
\]
hence
\[
\fint_\Omega (|x|+1)^\alpha\, \D x = 
\begin{cases}
0, & \alpha < 0 \\
1, & \alpha = 0 \\
\infty, & \alpha > 0
\end{cases}
\]
Thus $f(x) = (|x|+1)^\alpha$ is integrable if and only if $\alpha\leq 0$.
\end{ex}

\begin{ex}\label{ex:renint2}
Let again $\J = \R_+$ but this time ``$\preceq$'' $=$ ``$\geq$''.
Let $\Omega_T$ and $\mu_T$ be as above.
Then we have for any continuous function $f:\R \to \R$,
\[
\fint_\Omega f(x)\,\D x = \lim_{T\searrow 0}\frac{1}{2T}\int_{-T}^T f(x)\, dx = f(0).
\]
\end{ex}

\begin{ex}[Cauchy's principal value]
Let $\J = (0,1)$ and ``$\preceq$'' $=$ ``$\geq$''.
For $T\in\R_+$ let $\Omega_T = [-1,-T] \cup [T,1]$ and let $\mu_T$ be the usual Lebesgue measure.
Any measurable function $f:[-1,1]\to \R$ yields a measurable function on $\Omega$ by restriction.
Now the renormalized integral is nothing but Cauchy's principal value,
\[
\fint_\Omega f(x) \, \D x  = \lim_{T\searrow 0} \left(\int_{-1}^{-T} f(x)\,dx +\int_{T}^{1} f(x)\,dx\right) = \mathrm{CH} \int_{-1}^1 f(x)\,dx .
\]
\end{ex}

\begin{ex}[Determinant of operators on a Hilbert space]
Let $\H$ be a separable real Hilbert space.
Let $\J$ be the set of all finite-dimensional subspaces of $\H$ ordered by inclusion, ``$\preceq$'' $=$ ``$\subset$''.
Every $n$-dimensional subspace $H\subset \H$ inherits an $n$-dimensional Lebesgue measure $d^nx$.
We equip $H$ with the renormalized measure $\mu_H := \pi^{-n/2}\,d^nx$ and consider the measure space family $\Omega=\{(H,\mu_H)\}_{H\subset\H,\,\dim(H)<\infty}$.
We let $I_H : H \to \H$ be the inclusion and $P_H: \H \to H$ the orthogonal projection.

Let us consider a bounded positive self-adjoint linear operator $L$ on $\H$.
We assume that $L$ is of the form $L=\Id + A$ where $A$ is of trace class.
Then the determinant of $L$ is defined and satisfies
\[
\det(L) = \prod_{j=1}^\infty (1+\lambda_j)
\]
where $\lambda_j$ are the eigenvalues of $A$ repeated according to their multiplicity, see \cite[Thm.~XIII.106]{RS4}.
We order the eigenvalues such that $|\lambda_1| \geq |\lambda_2| \geq \cdots \to 0$.

If $H \subset \H$ is an $n$-dimensional subspace and $\mu_1, \ldots, \mu_n$ are the eigenvalues of $P_H \circ L \circ I_H$, then we compute
\begin{align*}
\int_{H}\exp(-(Lx,x))\,d\mu_{H}
&=
\int_{H}\exp(-(L\circ I_H(x),I_H(x)))\,d\mu_{H} \\
&=
\int_{H}\exp(-(P_H \circ L\circ I_H(x),x))\,d\mu_{H} \\
&=
\pi^{-n/2} \int_{\R^n} \exp(-\mu_1 x_1^2) \cdots \exp(-\mu_n x_n^2)\,dx_1 \cdots dx_n \\
&=
\pi^{-n/2} \prod_{j=1}^n \int_{-\infty}^\infty \exp(-\mu_j x^2)\, dx \\
&=
\frac{1}{\sqrt{\mu_1}\cdots \sqrt{\mu_n}} \\
&=
\det(P_{H}\circ L \circ I_{H})^{-1/2} .
\end{align*}
Let $\eps>0$.
Since $B \mapsto \det(\Id+B)$ is continuous on the ideal of trace-class operators, there is a constant $\delta>0$ such that $|\det(L)^{-1/2}-\det(\Id+B)^{-1/2}|<\eps$ for all trace-class operators $B$ with trace-class norm $\|A-B\|_1<\delta$.
Choose $n$ so large that 
\[
\|A-P_{H_n}\circ A \circ I_{H_n}\|_1 = \sum_{j=n+1}^\infty |\lambda_j| < \frac{\delta}{2} 
\]
where $H_n$ is the span of the first $n$ eigenvectors.
Now let $H \subset \H$ be a finite dimensional subspace which contains $H_n$. 
Write $H = H_n \oplus V$ where $V$ is the orthogonal complement of $H_n$ in $H$.
We compute

\begin{align*}
\|A-P_H \circ A \circ I_H\|_1
&\leq
\|A-P_{H_n} \circ A \circ I_{H_n}\|_1 + \|P_H \circ A \circ I_H-P_{H_n} \circ A \circ I_{H_n}\|_1 \\
&<
\frac{\delta}{2} + \|P_V \circ A \circ I_V\|_1 \\
&=
\frac{\delta}{2} + \|P_V \circ (A-P_{H_n} \circ A \circ I_{H_n}) \circ I_V\|_1 \\
&\leq
\frac{\delta}{2} + \|P_V\| \cdot \|A-P_{H_n} \circ A \circ I_{H_n}\|_1 \cdot \|I_V\| \\
&=
\frac{\delta}{2} + \|A-P_{H_n} \circ A \circ I_{H_n}\|_1 \\
&<
\delta .
\end{align*}
Here we have freely identified operators acting on closed subspaces of $\H$ with the operators on $\H$ extended by zero to the orthogonal complement. 
Hence $\det(L)^{-1/2}$ differs from $\det(\Id+P_H \circ A \circ I_H)^{-1/2} = \int_{H}\exp(-(Lx,x))\,d\mu_{H}$ by an error smaller than $\eps$.
This shows
\[
\fint_\Omega \exp((-Lx,x)) \D x = \det(L)^{-1/2} .
\]
\end{ex}

\begin{ex}[Fourier transform of $L^p$-functions]
Fix $n\in\N$.
Let $\J$ be the set of all compact subsets of $\R^n$ ordered by inclusion, ``$\preceq$'' $=$ ``$\subset$''.
For any $K\in\J$, the corresponding measure space is $K$ together with the $n$-dimensional Lebesgue measure $d^nx$ induced from $\R^n$.
Any measurable function on $f:\R^n\to\C$ yields a measurable function on $\Omega=\{(K,d^nx)\}_{K\in\J}$ by restriction.
If $f\in L^1(\R^n)$, then both the dominated convergence theorem and the monotone convergence theorem imply
\[
\fint_\Omega f(x)\D x = \int_{\R^n} f(x)\, d^nx .
\]
In this sense, the renormalized integral generalizes the usual integral in this example.
For integrable $f$ the Fourier transform $\hat f$ is defined by 
\begin{equation}
\hat f (x) = (2\pi)^{-n/2} \int_{\R^n} e^{-i\<x,y\>}f(y)\,d^ny .
\label{eq:FourierDef}
\end{equation}
Let $1<p\le 2$ and $q$ such that $1/p + 1/q = 1$.
The Hausdorff-Young inequality \cite[Thm.~IX.8]{RS2} states that
\[
\|\hat f \|_{L^q} \leq (2\pi)^{n/2-n/p} \|f\|_{L^p} 
\]
for all $f\in L^1(\R^n) \cap L^p(\R^n)$.
Hence Fourier transformation extends uniquely to a bounded linear map $L^p(\R^n) \to L^q(\R^n)$.
However, for general $f\in L^p(\R^n)$ the integral in the original definition \eqref{eq:FourierDef} no longer exists.
We show that it does exist as a renormalized integral.

For $K\in\J$, let $\chi_K:\R^n\to\R$ be the characteristic function of $K$, i.e., 
\[
\chi_K(x) = \begin{cases}
            1, & \mbox{ for }x\in K \\
            0, & \mbox{ for }x\notin K 
            \end{cases}
\]
If $K$ contains the ball with center $0$ and radius $R$, we have
\[
\int_{\R^n} |f(y)-\chi_K(y)f(y)|^p\, d^ny
=
\int_{\R^n} (1-\chi_K(y)) |f(y)|^p\, d^ny
\le
\int_{|y|\ge R}|f(y)|^p\, d^ny .
\]
The dominated convergence theorem shows that 
\[
\int_{|y|\ge R}|f(y)|^p\, d^ny \longrightarrow 0
\]
as $R\to\infty$.
This shows 
\[
\varinjlim_{K\in\J}\chi_Kf = f \,\,\mbox{ in }\,\, L^p(\R^n)
\] 
and hence 
\[
\varinjlim_{K\in\J}\widehat{\chi_Kf} = \hat f \,\,\mbox{ in }\,\, L^q(\R^n).
\]
By H\"older's inequality, $\chi_Kf\in L^1(\R^n)$.
Therefore
\[
\widehat{\chi_Kf}(x) = (2\pi)^{-n/2}\int_{K} e^{-i\<x,y\>}f(y)\,d^ny 
\]
and hence
\begin{equation}
\hat f 
\,\,=\,\,
\varinjlim_{K\in\J}(2\pi)^{-n/2}\int_{K} e^{-i\<\cdot,y\>}f(y)\,d^ny
\,\,=\,\,
(2\pi)^{-n/2}\fint_\Omega e^{-i\<\cdot,y\>}f(y) \D y .
\label{eq:Fourier}
\end{equation}
Strictly speaking, this example requires a refinement of Definition~\ref{def:integrierbar}.
The limit in \eqref{eq:Fourier} exists in $L^q(\R^n)$ but the integrands $e^{-i\<\cdot,y\>}f(y)$ for fixed $y$ are not in $L^q(\R^n)$.
Instead of having one Banach space $X$ we could require locally convex topological vector spaces $X_0$ and $X_1$, $X_0$ continuously embedded in $X_1$, such that the functions $f_\T$ take values in $X_1$ and the integrals converge in $X_1$ but the integrals are actually in $X_0$ and the directed limit exists in $X_0$.
In our example we can then choose $X_0=L^q(\R^n)$ and $X_1=L^1_\mathrm{loc}(\R^n)$.

For the sake of simplicity we will use the simpler version of renormalized integrals as given in Definition~\ref{def:integrierbar}.
\end{ex}

\begin{rem}
In general, renormalized integrals have all properties of conventional integrals which are preserved under limits.
Given $\Omega=\{(\Omega_\T,\mu_\T)\}_{\T\in\J}$ and $(X,\|\cdot\|)$ as above, we obviously have
\begin{itemize}
\item 
Linearity: 
The space of integrable functions $f$ on $\Omega$ with values in $X$ forms a vector space and
\[
\fint_\Omega (\alpha f(x)+\beta g(x))\, \D x = 
\alpha \fint_\Omega f(x)\, \D x + \beta \fint_\Omega g(x)\, \D x
\]
for all integrable $f$ and $g$ and all numbers $\alpha$ and $\beta$.
\item
Monotonicity:
If $X=\R$ and $f$ and $g$ are integrable with $f \leq g$,  i.e., if $f_\T \leq g_\T$ holds eventually, then
\[
\fint_\Omega f(x)\, \D x \leq \fint_\Omega g(x)\, \D x
\]
\item
Triangle inequality:
If $f$ and the pointwise norm of $f$ are integrable functions on $\Omega$, then
\[
\left\|\,\fint_\Omega f(x)\,\D x\,\right\| \leq \fint_\Omega \|f(x)\|\,\D x
\]
\end{itemize}
\end{rem}

\textbf{Warning.} 
In general, the monotone convergence theorem, the dominated convergence theorem, and the Fatou lemma do not hold for renormalized integrals.
In Example~\ref{ex:renint1} the functions $f_n(x) = (|x|+1)^{-1/n}$ form a sequence of positive integrable functions converging monotonically from below to the integrable function $f(x)=1$.
But for the integrals we have
\[
\lim_{n\to\infty}\fint_\Omega f_n(x)\,\D x = 0 < \fint_\Omega f(x)\,\D x = 1 .
\]
This violates all three of the above theorems.
This also shows that the renormalized integral in Example~\ref{ex:renint1} is not induced by a measure on $\R$.
In Example~\ref{ex:renint2} the situation is different because here the renormalized integral coincides with the conventional integral with respect to the Dirac measure supported at $0$.

\section{Path integrals on manifolds}
\label{sec:PathIntegrals}

By a {\em partition} we mean a finite sequence of increasing real numbers $\PP=(0=s_0 < s_1 < \cdots < s_r=1)$.
We think of $\PP$ as a subdivision of the interval $[0,1]$ into subintervals $[s_{j-1},s_j]$.
The {\em mesh} of $\PP$ is given by $|\PP| := \max_{j=1,\dots,r}|s_j-s_{j-1}|$.

The set of partitions $\PP$ forms a directed system.
Here $\PP \preceq \PP'$ if and only if $\PP'$ is a subdivision of $\PP$, i.e., $\PP$ is a subsequence of $\PP'$.

Let $M$ be a Riemannian manifold.
A {\em piecewise smooth curve} in $M$ is a pair $(\PP,\gamma)$ where $\PP$ is a partition and $\gamma : [0,1] \to M$ is a continuous curve whose restrictions to the subintervals $[s_{j-1},s_j]$ are smooth.
A piecewise smooth curve $(\PP,\gamma)$ is called a {\em geodesic polygon} if for every $j = 1, \ldots, r$ the point $\gamma(s_j)$ is not in the cut-locus of $\gamma(s_{j-1})$ and $\gamma|_{[s_{j-1},s_j]}$ is the unique shortest geodesic joining its endpoints.
Let $\PPM := \{(\PP,\gamma)\, |\,(\PP,\gamma)\mbox{ is a geodesic polygon}\}$ be the space of all geodesic polygons parameterized on the partition $\PP$.
Moreover, given $x,y\in M$, we put $\PPMx := {\{(\PP,\gamma)\in\PPM\,|\, \gamma(0)=x\}}$, $\PPMy := \{(\PP,\gamma)\in\PPM\,|\, \gamma(1)=y\}$, and $\PPMxy := \PPMx \cap \PPMy$. 

For a fixed partition $\PP$ any geodesic polygon $(\PP,\gamma)$ is uniquely determined by the tuple of vertices $(\gamma(s_0),\ldots,\gamma(s_r))$.
Hence $\PPM$ can be identified with the set $\{(x_0,\ldots,x_r)\in M\times \cdots\times M\,|\,x_j \mbox{ does not lie in the cut-locus of }x_{j-1}\mbox{ for all }j=1,\ldots,r\}$.
This is an open and dense subset of $M\times \cdots\times M=M^{\times(r+1)}$.
We write $(\PP,\gamma(x_0,\ldots,x_r))$ for the geodesic polygon parameterized on $\PP$ with vertices $\gamma(s_j)=x_j$.
Via this identification $\PPM$ inherits a measure induced by the Riemannian product volume measure on $M^{\times(r+1)}$.
Similarly, $\PPMx$, $\PPMy$ and $\PPMxy$ inherit measures from the Riemannian product volume measures on $\Mr$, $\Mr$ and $M^{\times(r-1)}$ respectively.
We denote these measures on $\PPM$, $\PPMx$, $\PPMy$, and on $\PPMxy$ by $\Dg$.

For any partition $\PP=(s_0 < s_1 < \cdots < s_r)$, for any $m\in\N$ and any $t>0$ we define the {\em renormalization constant} by
\[
Z(\PP,m,t) := \prod_{j=1}^r (4\pi t(s_j-s_{j-1}))^{m/2} = t^{rm/2} \prod_{j=1}^r (4\pi (s_j-s_{j-1}))^{m/2}.
\]
Fix $t>0$.
For each partition $\PP$ we now have a measure space $(\PPM,{Z(\PP,\dim(M),t)\hme\cdot\Dg})$.
Denote the measure space family $\{(\PPM,{Z(\PP,\dim(M),t)\hme\cdot \Dg})_\PP\}_\PP$ by $\PMt$.
The measure space families $\PMtx$, $\PMty$, and $\PMtxy$ are defined similarly.

\begin{definition}\label{def:pfadintegrierbar}
Let $(X,\|\cdot\|)$ be a Banach space.
If  $F=\{F_\PP\}_\PP$ is an integrable function on $\PMt$ with values in $X$ in the sense of Definition~\ref{def:integrierbar}, then we call $F$ {\em path integrable}.
We write 
\[
\fint_\PMt F(\gamma)\,\Dg
\]
for the value of the integral and call it the {\em value of the path integral}.
\end{definition}

There is a certain sloppiness in this notation because in general $F$ is actually a function of the pair $(\PP,\gamma)$, not of $\gamma$ alone. 

In the same way, one defines path integrals of functions on $\PMtx$, on $\PMty$, and on $\PMtxy$.

\begin{ex}\label{ex:qQ}
Let $\E(\gamma) = \frac12 \int_0^1 |\dot\gamma(t)|^2dt$ denote the \emph{energy} of $\gamma$.
The energy is defined for all piecewise smooth curves, in particular for geodesic polygons.
We will see that the function $F(\gamma) = \exp(-\E(\gamma)/2t)$ is path integrable on $\PMtxy$.
The value of the path integral 
\[
\fint_\PMtxy \exp\left(-\frac{1}{2t} \E(\gamma)\right) \Dg
\]
turns out to be the heat kernel of the operator $\Delta + \frac13 \scal$, evaluated at the points $x$ and $y$ and at time $t$.
Here $\Delta=\delta d$ is the Laplace-Beltrami operator and $\scal$ denotes scalar curvature.
\end{ex}

\begin{ex}\label{ex:kernel}
Let $E \to M$ be a vector bundle over $M$.
Denote by $E\boxtimes E^*\to M \times M$  the exterior tensor product whose fiber over $(x,y) \in M\times M$ is given by $(E\boxtimes E^*)_{(x,y)} = {E_x \otimes E^*_y} = \Hom(E_y,E_x)$.
Let $q(t,x,y)\in \Hom(E_y,E_x)$ depend continuously on $x,y\in M$ and $t>0$.
We call such a map $q$ a \emph{continuous time-dependent integral kernel} in $E$.

Such a kernel induces a function $Q$ on geodesic polygons by
\begin{align*}
Q_t(\PP,\gamma) 
:=&
q(t(s_r-s_{r-1}),\gamma(s_r),\gamma(s_{r-1})) 
\circ \cdots \circ q(t(s_1-s_{0}),\gamma(s_1),\gamma(s_{0}))  \\ 
&
\in \Hom(E_{\gamma(0)},E_{\gamma(1)}) .
\end{align*}
If we fix $x$ and $y\in M$, then $Q$ is a function on $\PMxy$ with values in the vector space $\Hom(E_x,E_y)$.
If $q$ has the semigroup property, i.e.,
\[
\int_M q(t,x,y)\circ q(t',y,z)\, dy = q(t+t',x,z)
\]
for all $x,z \in M$ and all $t,t'>0$, then
\begin{align*}
Z(\PP,\dim(M),t)&\hme\int_\PPMxy Z(\PP,\dim(M),t) Q_t(\PP,\gamma) \Dg \\
&=
\int_\PPMxy  Q_t(\PP,\gamma) \Dg \\
&=
\int_{M^{\times(r-1)}} q(t(s_r-s_{r-1}),y,z_{r-1}) \circ \cdots \circ q(t(s_1-s_{0}),z_1,x)\,dz_1\cdots dz_{r-1} \\
&=
q(t,y,x) .
\end{align*}
Thus the function $(\PP,\gamma) \mapsto Z(\PP,\dim(M),t)\,Q_t(\PP,\gamma)$ is path integrable in this case and 
\begin{equation}
\fint_\PMtxy Z(\PP,\dim(M),t)\,Q_t(\PP,\gamma)\,\Dg = q(t,y,x).
\label{eq:tautopath}
\end{equation}
\end{ex}
Functions of the form $Q_t$ where $q(t,x,y)$ does not have the semigroup property will be of central importance.
We need a criterion that ensures the path integrability of $Q_t$.

\begin{definition}\label{def:heatbound}
Let $M$ be a compact Riemannian manifold and let $E\to M$ be a Hermitian vector bundle.
A continuous time-dependent integral kernel $q$ in $E$ is said to \emph{satisfy a heat bound} if there exist positive constants $T, C, B_1,\ldots,B_k$ such that
\[
|q(t,x,y)| \,\leq\, \kD(t,x,y) + Ct\sum_{j=1}^k \kD(B_jt,x,y)
\]
for all $t\in(0,T]$ and $x,y\in M$.
Here $\kD$ denotes the heat kernel of the Laplace-Beltrami operator $\Delta$ on $M$.
\end{definition}

\begin{definition}\label{def:heatrelated}
Let $M$ be a compact Riemannian manifold, let $E\to M$ be a Hermitian vector bundle and let $q$ and $q'$ be continuous time-dependent integral kernels in $E$.
We say that $q$ and $q'$ are \emph{heat-related} if there exist positive constants $T, C, B_1,\ldots,B_k$ and $\beta>1$ such that
\[
|q(t,x,y)-q'(t,x,y)| \,\leq\,  Ct^\beta\sum_{j=1}^k \kD(B_jt,x,y)
\]
for all $t\in(0,T]$ and $x,y\in M$.
\end{definition}

We put
\begin{equation}
e(t,x,y) := (4\pi t)^{-m/2}\exp\left(-\frac{d(x,y)^2}{4t}\right)
\label{eq:defe}
\end{equation}
where $m=\dim(M)$.
This is a continuous time-dependent integral kernel in the trivial line bundle.
It generalizes the Gaussian normal distribution on $\R^m$ to manifolds.

Here is a criterion which will allow us in concrete situations to check that two kernels are heat-related.

\begin{lemma}\label{lem:gaussest}
Let $M$ be a compact Riemannian manifold, let $E \to M$ be a Hermitian vector bundle over $M$.
Let $q$ and $q'$ be continuous time-dependent integral kernels in $E$.
If there exist $C,\alpha,\beta \geq 0$ with $\beta+\alpha/2>1$ and $T>0$ such that 
\[
|q(t,x,y) - q'(t,x,y)| \,\leq\, C\cdot e(t,x,y) \cdot d(x,y)^\alpha \cdot t^\beta
\]
for all $(t,x,y)\in (0,T]\times M \times M$, then $q$ and $q'$ are heat-related.
\end{lemma}

\begin{proof}
We choose a constant $C_1>0$ such that $\tau^\alpha \leq C_1 \cdot \exp(\tau^2)$ for all $\tau \in [0,\infty)$.
With $\tau = {d(x,y)}/{\sqrt{8t}}$ this yields
\begin{equation}
d(x,y)^\alpha 
\,\leq\,
C_1 \cdot (8t)^{\alpha/2} \cdot \exp\left(\frac{d(x,y)^2}{8t}\right) .
\label{est:dxygegent}
\end{equation}
Hence
\begin{align}
|q(t,x,y) - q'(t,x,y)|
&\leq
C \cdot e(t,x,y)\cdot d(x,y)^\alpha \cdot t^\beta \nonumber\\
&\stackrel{(\ref{est:dxygegent})}{\leq}
C_2 \cdot e(t,x,y)\cdot t^{\beta+\alpha/2} \cdot
\exp\left(\frac{d(x,y)^2}{8t}\right)  \nonumber\\
&=
C_3 \cdot e(2t,x,y)\cdot t^{\beta+\alpha/2} .
\label{eq:workhorse1}
\end{align}
The heat kernel of the Laplace-Beltrami operator satisfies the well-known bound 
\begin{equation}\label{eq:hsu}
\kD(t,x,y) \,\geq\, C_4 \cdot e(t,x,y)
\end{equation}
for all $(t,x,y) \in (0,1] \times M \times M$, see e.g.\ \cite[Cor.~5.3.5]{H}.
Inserting \eqref{eq:hsu} into \eqref{eq:workhorse1} yields
\[
|q(t,x,y) - q'(t,x,y)|
\,\leq\, 
C_5 \cdot t^{\beta+\alpha/2} \cdot \kD(2t,x,y)
\]
which proves the claim.
\end{proof}

The following proposition shows why heat bounds on kernels are important for path integrals.

\begin{prop}\label{prop:replace}
Let $M$ be an $m$-dimensional compact Riemannian manifold, let ${E \to M}$ be a Hermitian vector bundle over $M$.
Let $q$ and $q'$ be continuous time-dependent integral kernels in $E$.
Let $t>0$.
Let $Q_t, Q_t' : \PMtxy \to \Hom(E_x,E_y)$ be the corresponding measurable functions.

\vspace{-3mm}
Suppose that $q$ satisfies a heat bound and that $Q_t$ is path integrable.
If $q$ and $q'$ are heat-related, then $q'$ also satisfies a heat bound, $Q_t'$ is also path integrable and the path integrals coincide,
\[
\fint_\PMtxy Z(\PP,m,t)Q_t(\PP,\gamma)\, \D\gamma =
\fint_\PMtxy Z(\PP,m,t)Q_t'(\PP,\gamma)\, \D\gamma.
\]
\end{prop}

\begin{proof}
Let $q$ and $q'$ be heat-related.
It is clear from the definitions that $q'$ also satisfies a heat bound.
We put $B_\mathrm{min} := \min\{1,B_1,\ldots,B_k\}$ and $B_\mathrm{max} := \max\{1,B_1,\ldots,B_k\}$ for the constants $B_j$ occurring in Definitions~\ref{def:heatbound} and \ref{def:heatrelated}.

Let $\PP$ be a partition whose mesh $\mu$ is sufficiently small so that the estimates in Definitions~\ref{def:heatbound} and \ref{def:heatrelated} apply.
Using the semigroup property of $\kD$ we estimate
\begin{align*}
&\int_{\PPMxy} \left|Q_t(\PP,\gamma) - Q_t'(\PP,\gamma) \right| \Dg\\
&=
\int_{M^{\times(r-1)}} \Big|\sum_{j=1}^r q(t(s_r-s_{r-1}),y,z_{r-1}) \circ \cdots \circ q(t(s_{j+1}-s_{j}),z_{j+1},z_{j})\circ\\
&
\quad\quad\quad\circ (q-q')(t(s_j-s_{j-1}),z_j,z_{j-1}) \circ q'(t(s_{j-1}-s_{j-2}),z_{j-1},z_{j-2})\circ\cdots \\
&
\quad\quad\quad\cdots\circ q'(t(s_1-s_{0}),z_1,x)\Big|\,dz_1\cdots dz_{r-1} \\
&\leq
\int_{M^{\times(r-1)}}\sum_{j=1}^r |q(t(s_r-s_{r-1}),y,z_{r-1})| \circ \cdots \\
&
\quad\quad\quad\cdots\circ |(q-q')(t(s_j-s_{j-1}),z_j,z_{j-1})| \circ\cdots \circ |q'(t(s_1-s_{0}),z_1,x)|\,dz_1\cdots dz_{r-1} \\
&\leq
\sum_{j=1}^r\int_{M^{\times(r-1)}} \Big(\kD(t(s_r-s_{r-1}),y,z_{r-1})+Ct(s_r-s_{r-1})\sum_{i_r=1}^k\kD(B_{i_r}t(s_r-s_{r-1}),y,z_{r-1})\Big) \cdot \\
&
\quad\quad\quad\cdots \Big(Ct^\beta(s_j-s_{j-1})^\beta\sum_{i_j=1}^k\kD(B_{i_j}t(s_j-s_{j-1}),z_j,z_{j-1})\Big)\cdots  \\
&
\quad\quad\quad\cdot \Big(\kD(t(s_1-s_{0}),z_1,x)+Ct(s_1-s_{0})\sum_{i_1=1}^k\kD(B_{i_1}t(s_1-s_{0}),z_1,x)\Big)\,dz_1\cdots dz_{r-1} \\
&\leq
\max_{s\in[B_\mathrm{min}t,B_\mathrm{max}t]} \kD(s,y,x) \cdot\sum_{j=1}^r(1+Ckt(s_r-s_{r-1})) \cdots Ckt^\beta(s_j-s_{j-1})^\beta \cdots (1+Ckt(s_1-s_{0})) \\
&\leq
\max_{s\in[B_\mathrm{min}t,B_\mathrm{max}t]} \kD(s,y,x)\cdot t^{\beta-1}\cdot \mu^{\beta-1} \cdot\sum_{j=1}^re^{Ckt(s_r-s_{r-1})} \cdots Ckt(s_j-s_{j-1}) \cdots e^{Ckt(s_1-s_{0})} \\
&\leq
\max_{s\in[B_\mathrm{min}t,B_\mathrm{max}t]} \kD(s,y,x)\cdot t^{\beta-1}\cdot \mu^{\beta-1} \cdot e^{Ckt} \cdot\sum_{j=1}^r Ckt(s_j-s_{j-1}) \\
&=
\max_{s\in[B_\mathrm{min}t,B_\mathrm{max}t]} \kD(s,y,x)\cdot t^{\beta}\cdot \mu^{\beta-1} \cdot e^{Ckt} \cdot Ck 
\end{align*}
The only term in this upper bound that depends on the partition is the term $\mu^{\beta-1}$.
Since $\beta>1$ this shows that
\[
\int_\PPMxy \left| Q_t(\PP,\gamma) -  Q_t'(\PP,\gamma)\right| \Dg 
\longrightarrow 0
\]
as $\mu \to 0$.
In the direct limit defining the path integral the mesh of the partitions tends to zero.
Thus the proposition is proved.
\end{proof}

\section{The heat kernel}

\subsection{Generalized Laplacians}
Throughout this section let $M$ be a compact $m$-dimensional Riemannian manifold without boundary and let $E\to M$ be a Hermitian vector bundle.
Let $H$ be a formally self-adjoint generalized Laplace operator acting on sections of $E$.
Locally, $H$ can be written in the form 
\[
H = -\sum_{j,k=1}^m g^{jk}\frac{\partial^2}{\partial x^j \partial x^k} + \mbox{lower order terms}.
\]
Here $(g^{jk})$ denotes the inverse of the matrix $(g_{jk})$ describing the Riemannian metric in the local coordinates, $g_{jk}=\<\partial/\partial x^j,\partial/\partial x^k\>$.
We assume that $H$ has smooth coefficients.
Formal self-adjointness means that for all smooth sections $u$ and $v$ in $E$,
\[
(Hu,v) = (u,Hv)
\]
holds, where $(u,v) = \int_M \< u(x),v(x)\>\,dx$ is the corresponding
$L^2$-scalar product.
Here $dx$ denotes the volume measure induced by the Riemannian metric.
It is well-known that $H$ is essentially self-adjoint in the Hilbert space
$L^2(M,E)$ of square-integrable sections in $E$ when given the domain
$C^\infty(M,E)$ of smooth sections in $E$, see e.~g.\ \cite[Prop.~2.33,
p.~89]{BGV}. 
Moreover, one knows that $H$ can be written in the form 
\begin{equation}
H = \nabla^*\nabla + V
\label{eq:H}
\end{equation}
where $\nabla$ is a metric connection on $E$ and $V$ is a smooth section in
symmetric endomorphisms of $E$, compare \cite[Prop.~2.5, p.~67]{BGV}.
We call $\nabla$ the \emph{connection determined by $H$} and $V$ its
\emph{potential}. 

\begin{ex}
The simplest example for $H$ as described above is the \emph{Laplace-Beltrami
operator} $H=\Delta$ acting on functions.
Here $E$ is the trivial real line bundle, $\nabla = d$ the usual derivative
and $V=0$.
\end{ex}

\begin{ex}
More generally, let $E=\bigwedge^k T^*M$ be the bundle of $k$-forms.
Then we may take the \emph{Hodge Laplacian} $H=d\delta + \delta d$ acting on
$k$-forms.
Here $d$ denotes exterior differentiation and $\delta$ its formal adjoint.
The Weitzenb\"ock formula says that $H = \nabla^*\nabla + V$, where
$\nabla$ is the Levi-Civita connection and $V$ depends linearly on the
curvature tensor of $M$. 
For example, for $k=1$ we have $V=\Ric$, see e.~g.\ \cite[Ch.~1.I]{Be}.
\end{ex}

\begin{ex}
If $M$ is a spin manifold one can form the spinor bundle $E=\Sigma M$ and the
Dirac operator $D$ acting on sections in $E$.
Then $H=D^2 = \nabla^*\nabla + \frac14 \scal$ is a self-adjoint generalized
Laplace operator.

More generally, the square of any generalized Dirac operator in the sense of
Gromov and Lawson yields a self-adjoint generalized Laplacian, see e.~g.\
\cite[Sec.~1,2]{GL}.
\end{ex}

\subsection{The heat kernel}
By functional calculus the self-adjoint extension of $H$ generates a
strongly continuous semigroup $t\mapsto e^{-tH}$ in the Hilbert space
$L^2(M,E)$. 
For $u \in L^2(M,E)$ the section $U(t,x) := (e^{-tH}u)(x)$, $(t,x) \in
[0,\infty)\times M$, is the unique solution to the heat equation
\[
\frac{\partial U}{\partial t} + HU =0
\]
satisfying the initial condition $U(0,x)=u(x)$.

For $t>0$ the operator $e^{-tH}$ is smoothing and has an integral kernel $\kH$, i.e.,
\[
e^{-tH}u(x) = \int_M \kH(t,x,y)\,u(y)\,dy .
\]
This integral kernel $(t,x,y)\mapsto \kH(t,x,y)$ is smooth on $(0,\infty)\times
M\times M$. 
It is called the \emph{heat kernel} for $H$.

The aim of this section is to give a path integral formula for this heat kernel.
Since the heat kernel has the semigroup property we have the tautological path integral formula as in \eqref{eq:tautopath}:
\begin{equation}
\kH(t,y,x) = \fint_\PMxy Z(\PP,\dim(M),t)\,K^H_t(\PP,\gamma)\,\Dg .
\label{eq:tautoheat}
\end{equation}
To turn this into something useful we will replace the heat kernel appearing in the definition of $K^H_t$ in the RHS of \eqref{eq:tautoheat} by heat-related continuous time-depend integral kernels (not having the semigroup property).
We will repeatedly use Proposition~\ref{prop:replace} and Lemma~\ref{lem:gaussest}.
To get started we need

\begin{lemma}
Let $M$ be a compact Riemannian manifold without boundary and let $E\to M$ be a Hermitian vector bundle.
Then the heat kernel of any formally self-adjoint generalized Laplace operator $H$ satisfies a heat bound.
\end{lemma}

\begin{proof}
Write the Laplace operator in the form $H=\nabla^*\nabla + V$.
Since $M$ is compact there exists a constant $C>0$ such that $V(x) \geq -C$ for all $x\in M$.
This means that all eigenvalues of the symmetric endomorphism $V(x)$ are bounded from below by $-C$.
By the \emph{Hess-Schrader-Uhlenbrock estimate}, see \cite[p.~32]{HSU}, we have
\[
|k^H(t,x,y)| \leq k^{\Delta-C}(t,x,y) = e^{Ct}\cdot\kD(t,x,y)
\]
for all $(t,x,y) \in (0,\infty) \times M \times M$.
For $t>0$ sufficiently small we have $e^{Ct} \leq 1 +2Ct$, which proves the heat bound.
\end{proof}

\subsection{First kernel modification}
For the first kernel modification we recall the heat kernel asymptotics.
Let $M\bowtie M := \{(x,y)\in M\,|\, x \mbox{ and } y \mbox{ are not cut-points}\}$.
Then $M\bowtie M$ is an open and dense subset of $M \times M$ containing the diagonal.
There are unique smooth sections $a_j$ of $E\boxtimes E$ over $M\bowtie M$ such that the formal heat kernel 
\[
e(t,x,y)\sum_{j=0}^\infty a_j(x,y)t^j
\]
formally solves the heat equation with respect to the $x$-variable,
\[
\left(\frac{\partial}{\partial t} + H_x\right) \left(e(t,x,y)\sum_{j=0}^\infty a_j(x,y)t^j\right) =0,
\]
and $a_0(x,x) = \Id_{E_x}$.
Here $e(t,x,y)$ is defined as in \eqref{eq:defe}.
For $N\in\N$ we get
\begin{equation}
\left(\frac{\partial}{\partial t} + H_x\right) \left(e(t,x,y)\sum_{j=0}^N a_j(x,y)t^j\right) 
=
e(t,x,y)\cdot H_xa_N(x,y) \cdot t^N
\label{eq:approx1}
\end{equation}
for $t\in(0,\infty)$ and $(x,y)\in M\bowtie M$.
See \cite[Thm.~2.26]{BGV} for details.
Pick $\eta>0$ such that $2\eta$ is smaller than the injectivity radius of $M$.
Choose a smooth cutoff function $\chi : \R \to \R$ such that 
\begin{itemize}
\item $\chi\equiv 1$ on $(-\infty,\eta]$
\item $\chi\equiv 0$ on $[2\eta,\infty)$
\item $0 \leq \chi \leq 1$ everywhere
\end{itemize}
We put 
\[
\kHN(t,x,y) := \chi(d(x,y))\cdot e(t,x,y)\cdot \sum_{j=0}^N a_j(x,y)t^j .
\]
Then $\kHN$ is smooth on all of $(0,\infty) \times M \times M$.
From \eqref{eq:approx1} we get
\[
\left(\frac{\partial}{\partial t} + H_x\right) \kHN(t,x,y)
=
e(t,x,y)\cdot \left(\chi(d(x,y))\cdot H_xa_N(x,y) \cdot t^N
+ b_N(t,x,y)\right)
\]
where the support of $b_N$ is contained in the region where the gradient of $\chi(d(x,y))$ does not vanish, i.e., in the region $(0,\infty) \times \{(x,y)\in M\times M\,|\,\eta\leq d(x,y) \leq 2\eta\}$.
Moreover, explicit computation shows 
\begin{equation}
b_N(t,x,y) = \OO(t^{-1}) \quad\mbox{ as }\quad t\searrow 0 
\label{eq:bNOO}
\end{equation}
uniformly in $x$ and $y$.
Duhamel's principle \cite[Prop.~7.9]{Roe} implies
\begin{align}
\kH&(t,x,y) - \kHN(t,x,y)\nonumber\\
&=
\int_0^t \int_M \kH(t-s,x,z)\left(e(s,z,y)\cdot \left(\chi(d(z,y))\cdot H_xa_N(z,y) \cdot s^N
+ b_N(s,z,y)\right)\right)dz\,ds
\label{eq:approx2}
\end{align}
Using the Hess-Schrader-Uhlenbrock inequality and \eqref{eq:hsu} we estimate for all $t\in(0,1]$ and $x,y\in M$
\begin{align}
\Big|\int_0^t \int_M \kH&(t-s,x,z)\cdot e(s,z,y)\cdot \chi(d(z,y))\cdot H_xa_N(z,y) \cdot s^N dz\,ds\Big|\nonumber\\
&\leq
\int_0^t \int_M \left|\kH(t-s,x,z)\right|\cdot  e(s,z,y) \cdot \chi(d(z,y))\cdot \left| H_xa_N(z,y)\right| \cdot s^N\, dz\,ds \nonumber\\
&\leq
C_1\int_0^t \int_M e^{C_2(t-s)}\cdot \kD(t-s,x,z)\cdot \kD(s,z,y)\cdot s^N\,dz\,ds\nonumber\\
&=
C_1\int_0^t e^{C_2(t-s)}\cdot \kD(t,x,y)\cdot s^N\,ds\nonumber\\
&\leq
C_3\cdot \kD(t,x,y) \cdot t^{N+1} .
\label{eq:approx3}
\end{align}

Using the Hess-Schrader-Uhlenbrock inequality, \eqref{eq:hsu}, \eqref{eq:bNOO}, and the fact that $b_N(s,z,y)$ vanishes whenever $d(z,y)\leq \eta$ we estimate

{\allowdisplaybreaks
\begin{align}
\Big|&\int_0^t \int_M \kH(t-s,x,z)\cdot e(s,z,y)\cdot b_N(s,z,y)\,dz\,ds\Big|\nonumber\\
&\leq
\int_0^t \int_M e^{C_2(t-s)}\cdot \kD(t-s,x,z)\cdot e(s,z,y)\cdot\left|b_N(s,z,y)\right|\,dz\,ds\nonumber\\
&\leq
C_4\cdot \int_0^t \int_M \kD(t-s,x,z)\cdot e(s,z,y)\cdot\left|b_N(s,z,y)\right|\,dz\,ds\nonumber\\
&=
C_4 \cdot \int_0^t \int_M \kD(t-s,x,z)\cdot e(t+s,z,y)\cdot\left(\frac{t+s}{s}\right)^{m/2}\cdot\exp\left(-\frac{d(z,y)^2t}{4(t+s)s}\right)\cdot\left|b_N(s,z,y)\right|\,dz\,ds\nonumber\\
&\leq
C_5 \cdot \int_0^t \int_M \kD(t-s,x,z)\cdot e(t+s,z,y)\cdot s^{-m/2}\cdot\exp\left(-\frac{d(z,y)^2}{8s}\right)\cdot\left|b_N(s,z,y)\right|\,dz\,ds\nonumber\\
&\leq
C_6 \cdot \int_0^t \int_M \kD(t-s,x,z)\cdot e(t+s,z,y)\cdot s^{-m/2-1}\cdot\exp\left(-\frac{\eta^2}{8s}\right)\,dz\,ds\nonumber\\
&\leq
C_7 \cdot \int_0^t \int_M \kD(t-s,x,z)\cdot e(t+s,z,y)\cdot s^N\,dz\,ds\nonumber\\
&\leq
C_8 \cdot \int_0^t \int_M \kD(t-s,x,z)\cdot \kD(t+s,z,y)\cdot s^N\,dz\,ds\nonumber\\
&=
C_8 \cdot \int_0^t \kD(2t,x,y)\cdot s^N\,ds\nonumber\\
&=
C_9\cdot \kD(2t,x,y) \cdot t^{N+1} .
\label{eq:approx4}
\end{align}
}
Inserting \eqref{eq:approx3} and \eqref{eq:approx4} into \eqref{eq:approx2} yields
\[
\left|\kH(t,x,y) - \kHN(t,x,y)\right|
\leq
C_{10}\cdot (\kD(t,x,y)+\kD(2t,x,y)) \cdot t^{N+1}.
\]
This shows that $\kH$ and $\kHN$ are heat-related if $N\geq1$.
We use this with $N=1$.
Putting 
\[
\ki(t,x,y) := \kH_{(1)}(t,x,y) =  \chi(d(x,y))\cdot e(t,x,y)\cdot (a_0(x,y)+a_1(x,y)t)
\]
we have shown
\begin{lemma}
Let $M$ be a compact Riemannian manifold without boundary, let $E\to M$ be a Hermitian vector bundle and let $H$ be a formally self-adjoint generalized Laplacian acting on sections of $E$.

\vspace{-4mm}
Then the heat kernel $\kH$ and the smooth time-dependent integral kernel $\ki$ are heat-related.
In particular, $\ki$ satisfies a heat bound, $K^1_t$ is path integrable and 
\[
\kH(t,y,x) = \fint_\PMtxy Z(\PP,\dim(M),t)\,K^1_t(\PP,\gamma)\,\Dg .\tag*{\qed}
\]
\end{lemma}

\subsection{Second kernel modification}
If we put 
\[
a(x,y) := a_0(x,y)^{-1} \circ a_1(x,y) \in \Hom(E_y,E_y)
\]
then the integral kernel $\ki$ can written as
\[
\ki(t,x,y) = \chi(d(x,y))\cdot e(t,x,y)\cdot a_0(x,y)\circ (\id + t a(x,y)).
\]
We set 
\[
\kii(t,x,y) := \chi(d(x,y))\cdot e(t,x,y)\cdot a_0(x,y)\circ \exp(t a(x,y)).
\]
Since $\exp(t a(x,y)) - (\id + t a(x,y)) = \OO(t^2)$ uniformly in $x$ and $y$ with $d(x,y)\leq 2\eta$ we have
\begin{lemma}
Let $M$ be a compact Riemannian manifold without boundary, let $E\to M$ be a Hermitian vector bundle and let $H$ be a formally self-adjoint generalized Laplacian acting on sections of $E$.

\vspace{-4mm}
Then the smooth time-dependent integral kernels $\ki$ and $\kii$ are heat-related.
In particular, $\kii$ satisfies a heat bound, $K^2_t$ is path integrable and 
\[
\kH(t,y,x) = \fint_\PMtxy Z(\PP,\dim(M),t)\,K^2_t(\PP,\gamma)\,\Dg .
\tag*{\qed}
\]
\end{lemma}

\subsection{Third kernel modification}
For a piecewise smooth curve $(\PP,\gamma)$ and $s,t \in [0,1]$ and a connection $\nabla$ on $E$ let $\tau(\gamma,\nabla)_s^t : E_{\gamma(s)} \to E_{\gamma(t)}$ denote the \emph{parallel transport} along $\gamma$ with respect to $\nabla$.
We have 
\begin{equation}
\tau(\gamma,\nabla)_t^u \circ \tau(\gamma,\nabla)_s^t = \tau(\gamma,\nabla)_s^u
\quad\mbox{ and }\quad
\tau(\gamma,\nabla)_t^s = (\tau(\gamma,\nabla)_s^t)^{-1}.
\label{eq:paralleltransport} 
\end{equation}
We will use the metric connection $\nabla$ corresponding to a generalized Laplacian as in \eqref{eq:H}.
Then $\tau(\gamma,\nabla)_s^t$ is a linear isometry.

For $x$ and $y$ with $d(x,y) \leq 2\eta$ we define
\[
\kiii(t,x,y) := \chi(d(x,y))\cdot e(t,x,y)\cdot a_0(x,y)\circ \exp\Big(t\cdot\int_0^1 \tau(\gamma,\nabla)_s^1 \circ a(\gamma(s),\gamma(s))\circ \tau(\gamma,\nabla)_1^s\,ds\Big).
\]
Here $\gamma:[0,1]\to M$ denotes the shortest geodesic with $\gamma(0)=x$ and $\gamma(1)=y$.
This shortest geodesic is unique because $d(x,y)$ is smaller than the injecitivity radius of $M$.
For $d(x,y) > 2\eta$ set $\kiii(t,x,y) := 0$.

A proof similar to the one of \cite[Lemma~4.6]{BP} shows
\[
\left|\kii(t,x,y) - \kiii(t,x,y)\right|
\leq
C \cdot e(t,x,y) \cdot d(x,y) \cdot t.
\]
Hence Lemma~\ref{lem:gaussest} says that $\kii$ and $\kiii$ heat-related.
Proposition~\ref{prop:replace} applies and yields

\begin{lemma}
Let $M$ be a compact Riemannian manifold without boundary, let $E\to M$ be a Hermitian vector bundle and let $H$ be a formally self-adjoint generalized Laplacian acting on sections of $E$.

\vspace{-3mm}
Then the smooth time-dependent integral kernels $\kii$ and $\kiii$ are heat-related.
In particular, $\kiii$ satisfies a heat bound, $K^3_t$ is path integrable and 
\[
\kH(t,y,x) = \fint_\PMtxy Z(\PP,\dim(M),t)\,K^3_t(\PP,\gamma)\,\Dg .
\tag*{\qed}
\]
\end{lemma}

The advantage of $\kiii$ over $\kii$ consists of the fact that we need to evaluate $a_1$ only along the diagonal.
It is well-known that 
\[
a(x,x) = a_0(x,x)^{-1}\circ a_1(x,x) = a_1(x,x) = \frac16 \scal(x)\cdot \id_{E_x} - V(x)
\]
where $\scal$ denotes the scalar curvature of $M$ and $V$ is the potential of $H$; compare \cite[p.~103ff]{Roe}.
Hence $\kiii$ is given by
\begin{align*}
\kiii(t,x,y)
=&
\chi(d(x,y))\cdot e(t,x,y)\cdot a_0(x,y)\\
&
\circ \exp\left(t\cdot\int_0^1\Big(\frac16 \scal(\gamma(s))\cdot \id_{E_{y}} - \tau(\gamma,\nabla)_s^1 \circ V(\gamma(s))\circ \tau(\gamma,\nabla)_1^s\Big)\,ds\right).
\end{align*}

\subsection{Fourth kernel modification}
We can now replace $a_0(x,y)$ by another scalar curvature term.
The same estimates as in \cite[Section~4.5]{BP} show that $\kiii$ and $\kiv$ are heat-related, where
\begin{align*}
\kiv(t,x,y)
:=&
\chi(d(x,y))\cdot e(t,x,y)\cdot \tau(\gamma,\nabla)_1^0\\
&
\circ \exp\left(t\cdot\int_0^1\Big(\frac13 \scal(\gamma(s))\cdot \id_{E_{y}} - \tau(\gamma,\nabla)_s^1 \circ V(\gamma(s))\circ \tau(\gamma,\nabla)_1^s\Big)\,ds\right).
\end{align*}

\begin{lemma}\label{lem:lastmod}
Let $M$ be a compact Riemannian manifold without boundary, let $E\to M$ be a Hermitian vector bundle and let $H$ be a formally self-adjoint generalized Laplacian acting on sections of $E$.

\vspace{-3mm}
Then the smooth time-dependent integral kernels $\kiii$ and $\kiv$ are heat-related.
In particular, $\kiv$ satisfies a heat bound, $K^4_t$ is path integrable and 
\[
\kH(t,y,x) = \fint_\PMtxy Z(\PP,\dim(M),t)\,K^4_t(\PP,\gamma)\,\Dg .
\tag*{\qed}
\]
\end{lemma}
We can rewrite $\kiv$ in the form
\begin{align*}
\kiv(t,x,y)
=&
\chi(d(x,y))\cdot e(t,x,y)\cdot \exp\left(\frac{t}{3}\cdot\int_0^1 \scal(\gamma(s))\,ds\right)\cdot\tau(\gamma,\nabla)_1^0\\
&
\circ \exp\left(-t\cdot\int_0^1\left(\tau(\gamma,\nabla)_s^1 \circ V(\gamma(s))\circ \tau(\gamma,\nabla)_1^s\right)\,ds\right).
\end{align*}

\subsection{Path integral formula for the heat kernel}
We now come to the main result of this section.

\begin{definition}
Let $W$ be a continuous section of the endomorphism bundle $\Hom(E,E) = E\otimes E^* \to M$.
Let $\nabla$ be a connection on $E$.
For any piecewise smooth curve $(\PP,\gamma)$ in $M$ with $\PP=(0=s_0 < s_1 < \cdots < s_r=1)$ we define the \emph{$(\PP,\gamma)$-ordered exponential} by
\begin{align*}
\Pexp\Bigg(\int_{(\PP,\gamma)}&W\Bigg) \\
:=&
\prod_{j=1}^r \tau(\gamma,\nabla)_{s_{j-1}}^{s_j} \circ \exp\left(\int_{s_{j-1}}^{s_j}\left(\tau(\gamma,\nabla)_s^{s_{j-1}} \circ W(\gamma(s))\circ \tau(\gamma,\nabla)_{s_{j-1}}^s\right)\,ds\right)\\
=\,&
\tau(\gamma,\nabla)_{s_{r-1}}^{s_r} \circ \exp\left(\int_{s_{r-1}}^{s_r}\left(\tau(\gamma,\nabla)_s^{s_{r-1}} \circ W(\gamma(s))\circ \tau(\gamma,\nabla)_{s_{r-1}}^s\right)\,ds\right)
\circ\cdots\\
&
\cdots \circ \tau(\gamma,\nabla)_{s_{0}}^{s_1} \circ
\exp\left(\int_{s_{0}}^{s_1}\left(\tau(\gamma,\nabla)_s^{s_0} \circ W(\gamma(s))\circ \tau(\gamma,\nabla)_{s_0}^s\right)\,ds\right)\\
=\,&
\tau(\gamma,\nabla)_{0}^{1} \circ \exp\left(\int_{s_{r-1}}^{s_r}\left(\tau(\gamma,\nabla)_s^{0} \circ W(\gamma(s))\circ \tau(\gamma,\nabla)_{0}^s\right)\,ds\right)
\circ\cdots\\
&
\cdots \circ 
\exp\left(\int_{s_{0}}^{s_1}\left(\tau(\gamma,\nabla)_s^{0} \circ W(\gamma(s))\circ \tau(\gamma,\nabla)_{0}^s\right)\,ds\right)
\end{align*}
\end{definition}
where the last equation follows from \eqref{eq:paralleltransport}.
Note that $\Pexp\left(\int_{(\PP,\gamma)}W\right)\in\Hom(E_{\gamma(0)},E_{\gamma(1)})$.
If all $\tau(\gamma,\nabla)_s^1 \circ W(\gamma(s))\circ \tau(\gamma,\nabla)_1^s$ commute with each other, then 
\[
\Pexp\left(\int_{(\PP,\gamma)}W\right)
=
\tau(\gamma,\nabla)_{0}^{1} \circ \exp\left(\int_0^1\left(\tau(\gamma,\nabla)_s^0 \circ W(\gamma(s))\circ \tau(\gamma,\nabla)_0^s\right)\,ds\right).
\]
This is the case e.g.\ if $W$ is scalar, i.e., $W(x)=w(x)\cdot \id_{E_x}$ with $w(x)\in\R$.
Otherwise, $\Pexp\left(\int_{(\PP,\gamma)}W\right)$ depends on the subdivision $\PP$.

\begin{thm}\label{thm:heatkernel}
Let $M$ be a compact Riemannian manifold without boundary, let $E\to M$ be a Hermitian vector bundle and let $H$ be a formally self-adjoint generalized Laplacian acting on sections of $E$.
Let $\nabla$ be the connection determined by $H$ and $V$ its potential.

\vspace{-3mm}
Then the heat kernel of $H$ can be written as a path integral as follows:
\begin{align*}
\kH(&t,y,x) \\
&= \fint_\PMtxy \Xi(\PP,\gamma)  \cdot \exp\left(-\frac{\E[\gamma]}{2t} + \frac{t}{3}\int_0^1 \scal(\gamma(s))\,ds\right) \cdot \Pexp\left(\int_{(\PP,\gamma)}-tV\right) \,\Dg .\nonumber\\
\end{align*}
\end{thm}

\begin{proof}
We compute the integrand in the path integral formula for $\kH$ from Lemma~\ref{lem:lastmod}.
\begin{align}
K^4_t(\PP,\gamma)
=\,&
\kiv(t(s_r-s_{r-1}),\gamma(s_r),\gamma(s_{r-1})) 
\circ \cdots \circ \kiv(t(s_1-s_{0}),\gamma(s_1),\gamma(s_{0}))\nonumber\\
=&
\prod_{j=1}^r \chi(d(\gamma(s_j),\gamma(s_{j-1}))) \cdot
\prod_{j=1}^r e(t(s_j-s_{j-1}),\gamma(s_j),\gamma(s_{j-1}))\nonumber\\
&
\times\, \exp\left(\frac{t}{3}\int_0^1 \scal(\gamma(s))\,ds\right) \cdot\tau(\gamma,\nabla)_1^0 \circ \Pexp\left(\int_{(\PP,\gamma)}-tV\right)\nonumber\\
=\,&
\Xi(\PP,\gamma) \cdot Z(\PP,\dim(M),t)^{-1} \cdot \exp\left(-\sum_{j=1}^r\frac{d(\gamma(s_j),\gamma(s_{j-1}))^2}{4t(s_j-s_{j-1})}\right) \nonumber\\
&
\times\, \exp\left(\frac{t}{3}\int_0^1 \scal(\gamma(s))\,ds\right)\cdot \tau(\gamma,\nabla)_1^0 \circ \Pexp\left(\int_{(\PP,\gamma)}-tV\right)
\label{eq:PfadIntegrand}
\end{align}
Since $\gamma$ is a geodesic when restricted to one of the subintervals $[s_{j-1},s_j]$ it is parameterized proportionally to arclength, so that
\[
|\dot{\gamma}(s)| = \frac{d(\gamma(s_{j-1},\gamma(s_j))}{s_j-s_{j-1}},
\]
for all $s\in[s_{j-1},s_j]$.
Thus the energy of $\gamma|_{[s_{j-1},s_j]}$ is given by 
\[
\E[\gamma|_{[s_{j-1},s_j]}] 
=
\frac12 \int_{s_{j-1}}^{s_j} |\dot{\gamma}(s)|^2ds
=
\frac12 \frac{d(\gamma(s_{j-1},\gamma(s_j))^2}{s_{j-1}-s_j}.
\]
Hence the energy of $\gamma:[0,1] \to M$ is given by
\[
\E[\gamma] = \frac12 \sum_{j=1}^r \frac{d(\gamma(s_{j-1},\gamma(s_j))^2}{s_{j-1}-s_j}.
\]
Inserting this into \eqref{eq:PfadIntegrand} yields
\begin{align*}
K^4_t(\PP,\gamma)
=&
\Xi(\PP,\gamma) \cdot Z(\PP,\dim(M),t)^{-1} \cdot \exp\left(-\frac{1}{2t}\E[\gamma]\right) \nonumber\\
&
\times\, \exp\left(\frac{t}{3}\int_0^1 \scal(\gamma(s))\,ds\right)\cdot \tau(\gamma,\nabla)_1^0 \circ \Pexp\left(\int_{(\PP,\gamma)}-tV\right)
\end{align*}
Lemma~\ref{lem:lastmod} concludes the proof.
\end{proof}

\begin{cor}
Let $M$, $E$, $H$, $\nabla$, and $V$ be as in Theorem~\ref{thm:heatkernel}.
Suppose in addition that the potential $V$ is scalar, i.e., $V(x) = v(x)\cdot \id_{E_x}$ for a smooth function $v:M\to \R$.

\vspace{-2mm}
Then the heat kernel of $H$ can be written as a path integral as follows:
\begin{align*}
\kH&(t,y,x) \\ 
&= \fint_\PMtxy \Xi(\PP,\gamma)  \cdot \exp\left(-\frac{\E[\gamma]}{2t} + t\int_0^1\Big(\frac{1}{3} \scal(\gamma(s))-v(\gamma(s))\Big)\,ds\right) \cdot \tau(\gamma,\nabla)_0^1 \,\Dg .
\tag*{\qed}
\end{align*}
\end{cor}


\begin{thebibliography}{11}

\bibitem{AD}
{\sc L.~Andersson and B.~Driver:}
{\em Finite-dimensional approximations to Wiener measure and path integral formulas on manifolds.}
J.~Funct.\ Anal.~{\bf 165} (1999), 430--498. 

\bibitem{BP}
{\sc C.~B\"ar and F.~Pf\"affle:}
{\em Path integrals on manifolds by finite-dimensional approximation}.
J.~Reine Angew.~Math.~{\bf 625} (2008), 29--57.

\bibitem{BP2}
{\sc C.~B\"ar and F.~Pf\"affle:}
{\em Wiener Measures on Riemannian Manifolds and the Feynman-Kac Formula}.
\texttt{http://arxiv.org/abs/1108.5082}

\bibitem{BGV} 
{\sc N.~Berline, E.~Getzler, M.~Vergne:} 
{\em Heat Kernels and Dirac Operators.} 
Springer-Verlag, Berlin, 1992. 

\bibitem{Be}
{\sc A.~L.~Besse:}
{\em Einstein manifolds.}
Springer-Verlag, Berlin, 1987.

\bibitem{Bismut}
{\sc J.-M.~Bismut:}
{\em M\'ecanique al\'eatoire.}
Lecture Notes in Mathematics, 866.
Springer-Verlag, Berlin, 1981.

\bibitem{Elworthy}
{\sc K.~D.~Elworthy:}
{\it Stochastic differential equations on manifolds.}
Cambridge University Press, Cambridge, 1982.

\bibitem{Emery}
{\sc M.~\'Emery:}
{\it Stochastic calculus in manifolds.}
Springer-Verlag, Berlin, 1989.

\bibitem{Gangolli}
{\sc R.~Gangolli:}
{\it On the construction of certain diffusions on a differentiable manifold.}
Z.\ Wahrscheinlichkeitstheorie und Verw.\ Gebiete~\textbf{2} (1964), 406--419.

\bibitem{GL}
{\sc M.~Gromov, H.~B.~Lawson:}
{\em Positive scalar curvature and the Dirac operator on complete Riemannian
  manifolds.}
Inst.\ Hautes \'Etudes Sci.\ Publ.\ Math.\ {\bf 58} (1983), 83--196. 

\bibitem{HT}
{\sc W.~Hackenbroch and A.~Thalmaier:}
{\it Stochastische Analysis.}
Teubner, Stuttgart, 1994.

\bibitem{HSU}
{\sc H.~Hess, R.~Schrader, D.~A.~Uhlenbrock:}
{\em Kato's inequality and the spectral distribution of Laplacians on compact Riemannian manifolds.}
J.\ Diff.\ Geom.\ {\bf 15}  (1980), 27--37.

\bibitem{H}
{\sc E.~Hsu:}
{\em Stochastic Analysis on Manifolds}.
American Mathematical Society, Providence, Rhode Island, 2002

\bibitem{JL}
{\sc G.~W.~Johnson, M.~L.~Lapidus:}
{\em The Feynman integral and Feynman's operational calculus.}
Oxford University Press, Oxford, 2000.

\bibitem{RS2}
{\sc M.~Reed and B.~Simon:}
{\em Methods of Modern Mathematical Physics II - Fourier Analysis, Self-Adjointness}.
Academic Press, San Diego - New York - London, 1975

\bibitem{RS4}
{\sc M.~Reed and B.~Simon:}
{\em Methods of Modern Mathematical Physics IV - Analysis of Operators}.
Academic Press, San Diego - New York - London, 1978

\bibitem{Roe}
{\sc J.~Roe:}
{\em Elliptic operators, topology and asymptotic methods} (2nd edition).
Longman, Harlow, 1998.

\bibitem{SWW03} 
{\sc O.~G.~Smolyanov, H.~v.~Weizs\"acker, O.~Wittich:} 
{\em Chernoff's theorem and the construction of semigroups.} 
Evolution equations: applications to physics, industry, life sciences and
economics (Levico Terme, 2000), 349--358, Progr.~Nonlinear Differential
Equations Appl.~ 55, Birkh\"auser, Basel, 2003. 

\bibitem{SWW07} 
{\sc O.~G.~Smolyanov, H.~v.~Weizs\"acker, O.~Wittich:} 
{\em Chernoff's theorem and discrete time approximations of Brownian motion on manifolds.} 
Potential Anal.\ {\bf 26} (2007), 1--29. 

\bibitem{Stroock}
{\sc D.~W.~Stroock:}
{\it An introduction to the analysis of paths on a Riemannian manifold.}
American Mathematical Society, Providence, 2000.


\end{thebibliography}
\end{document}